\documentclass[
final,
nomarks
]{dmtcs-episciences}

\usepackage[utf8]{inputenc}
\usepackage{subfigure}

\usepackage{amsbsy}
\usepackage{amsmath,amsfonts,amssymb,amsthm,mathrsfs,mathtools}

\usepackage{bm}

\usepackage[round]{natbib}

\newcommand{\ndN}{\mathbb{N}}

\renewcommand{\Pr}[1]{\mathbb{P}(#1)}

\newcommand{\Ex}[1]{\mathbb{E}[#1]}

\newcommand{\cA}{\mathcal{A}}
\newcommand{\cB}{\mathcal{B}}
\newcommand{\cC}{\mathcal{C}}

\newcommand{\cE}{\mathcal{E}}
\newcommand{\cF}{\mathcal{F}}
\newcommand{\cG}{\mathcal{G}}

\newcommand{\cS}{\mathcal{S}}
\newcommand{\cT}{\mathcal{T}}
\newcommand{\cU}{\mathcal{U}}

\newcommand{\mA}{\mathsf{A}}
\newcommand{\mB}{\mathsf{B}}
\newcommand{\mC}{\mathsf{C}}

\newcommand{\mF}{\mathsf{F}}

\newcommand{\mT}{\mathsf{T}}

\newtheorem{theorem}{Theorem}[section]

\newtheorem{proposition}[theorem]{Proposition}
\newtheorem{lemma}[theorem]{Lemma}

\author[Benedikt Stufler]{Benedikt Stufler}
\title[Probabilistic enumeration and equivalence of nonisomorphic trees]{Probabilistic enumeration and equivalence of nonisomorphic trees}
\affiliation{
  Vienna University of Technology, Vienna, Austria}
\begin{document}

\publicationdata{vol. 27:3}{2025}{18}{10.46298/dmtcs.14790}{2024-11-18; 2024-11-18; 2025-10-14}{2025-10-16}
\maketitle
\begin{abstract}
	We present a new probabilistic proof of Otter's asymptotic formula for the number of unlabelled trees with a given number of vertices.  We additionally prove a new approximation result, showing that the total variation distance between random P\'olya trees and random unlabelled trees tends to zero when the number of vertices tends to infinity. In order to demonstrate that our approach is not restricted to trees we extend our results to tree-like classes of graphs.
\end{abstract}

\section{Introduction}

A \emph{tree} is a connected graph without cycles. Cayley's  theorem states that the number $u_n$ of trees with a given $n$-element vertex set satisfies 
\begin{align}
	\label{eq:labelled}
	u_n = n^{n-2}.
\end{align}
Numerous proofs of this quantity are known~\cite{MR1579119,MR0214515,MR633783,MR1673928,MR3617364,MR260623,10.1214/23-ECP523}. A common idea is to mark a root vertex. This way, each tree with $n$ vertices corresponds to precisely $n$ rooted versions, and the exponential generating series $T(z) = \sum_{n \ge 1} \frac{n u_n}{n!}  z^n$ of rooted trees satisfies 
\[
T(z) =  z \exp\left(	 T(z)\right) \qquad \text{and} \qquad T(1/e) = 1.
\]

Two trees are \emph{isomorphic} if there exists a bijection between their sets of vertices such that any two vertices are adjacent if and only if their images are adjacent. If each of the trees is endowed with a root vertex we additionally require that each  root vertex is the image of the other under this bijection. This gives rise to two different numbers of nonisomorphic or unlabelled trees with a given number $n \ge 1$ of vertices: The number $f_n$ of  unlabelled (unrooted) trees, also called free trees, and the number $a_n$ of unlabelled rooted trees, also called P\'olya trees.   Contrarily to the labelled case, different unlabelled trees with the same number of vertices may correspond to different numbers of rooted versions. See Figure~\ref{fi:polya}.

\begin{figure}[t]
	\centering
	\begin{minipage}{\textwidth}
		\centering
		\includegraphics[width=0.35\linewidth]{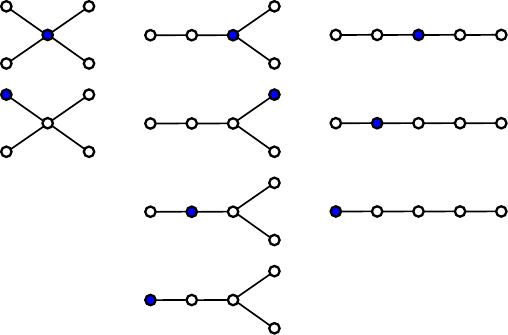}
		\caption[asdf]{Rooted unlabelled trees with five vertices grouped according to the underlying unrooted tree.}
		\label{fi:polya}
	\end{minipage}
\end{figure}

In order to tackle such challenges, P\'olya~\cite{MR1577579} developed a general theory for counting objects up to isomorphism. In particular, he reformulated Cayley's functional equation for the ordinary generating series $A(z) = \sum_{n \ge 1} a_n z^n$ to
\begin{align*}
	A(z) = z \exp \left( \sum_{i \ge 1} A(z^i)/i \right).
\end{align*}
Later, Otter~\cite{MR0025715} derived the asymptotic  formula
\begin{align}
	\label{eq:polya}
	a_n \sim c_A n^{-3/2} \rho^{-n}
\end{align}
for constants $c_A \approx 0.439924$ and $\rho \approx 0.338321$, and the fact
\begin{align*}
	A(\rho) = 1.
\end{align*}
A full asymptotic expansion was derived by Genitrini~\cite{MR3817525} using methods of analytic combinatorics~\cite{MR2483235,MR1039294}. Otter additionally obtained the asymptotic number of unlabelled trees. Our first contribution in the present work is to provide a new probabilistic proof of his famous result:
\begin{theorem}[Otter~\cite{MR0025715}]
	\label{te:main1}
	Setting $c_F = 2 \pi c_A^3$, we have as $n$ tends to infinity 
	\begin{align}
		\label{eq:free}
		f_n \sim c_F n^{-5/2} \rho^{-n}.
	\end{align}
\end{theorem}
Otter's method is based on establishing the dissymmetry  equation for the ordinary generating series $F(z) = \sum_{n \ge 1} f_n z^n$
\begin{align*}
	F(z) = A(z) - \frac{1}{2}(A(z)^2 - A(z^2)),
\end{align*}
from which the asymptotic expansion of $f_n$ may readily be deduced. Alternative proofs of the dissymmetry equation were given by Harary~\cite{MR78687} by using directly equations by P\'olya~\cite{MR1577579}, and by Bodirsky, Fusy, Kang and Vigerske~\cite{MR2810913} by introducing the cycle pointing method. Another proof of Equation~\eqref{eq:free} was given in~\cite[Thm. 13]{MR2873207} using analytic integration and singularity analysis. 
The proof of Theorem~\ref{te:main1} we present here is based on probabilistic methods and  does not recover the dissymmetry equation.

The development of methods for asymptotically counting objects up to symmetry is of general interest since there are clear problems in enumerative combinatorics that remain open and pose a serious challenge. A prominent example is the asymptotic enumeration of unlabelled planar graphs, that is, graphs  that may be drawn in the plane so that edges only intersect at endpoints. Liskovet and Walsh~\cite{MR945238} formulated a program for their enumeration. Bender, Gao and Wormald~\cite{MR1946145} established the asymptotic number of labelled $2$-connected planar graphs, and the asymptotic number of connected and unrestricted labelled planar graphs was then obtained in the breakthrough result by Gim\'enez and Noy~\cite{MR2476775}. An alternative combinatorial approach was later given by Chapuy, Fusy, Kang and Shoilekova~\cite{MR2465772}, and a third probabilistic proof via large deviation methods was  given by S.~\cite{zbMATH07665039}. Determining the asymptotic number of nonisomorphic planar graphs with a given number of vertices however remains a major challenge that requires the development of new methods and approaches.

Apart for enumerating combinatorial structures, we also want to understand their typical shape when generated at random. When studying the uniform random free tree~$\mF_n$ with $n$ vertices the main strategy is to prove properties of the random $n$-vertex P\'olya tree $\mA_n$ instead and then transfer the results.  The trivial fact that any free tree with $n$ vertices has at least one and at most $n$ rooted version entails that for any set $\cE$ of free trees
\begin{align}
	\label{eq:rough}
	\Pr{F(\mA_n) \in \cE} \frac{a_n}{n f_n } \le \Pr{\mF_n \in \cE} \le \Pr{F(\mA_n) \in \cE} \frac{a_n}{f_n}.
\end{align}
Here $F(\mA_n)$ denotes the unrooted tree obtained by forgetting which vertex of $\mA_n$ is marked. Hence, the first order asymptotics for $a_n$ and $f_n$ imply that any property that holds with probability $o(1/n)$ for $F(\mA_n)$ holds with probability $o(1)$ for $\mF_n$.

This crude bound may serve in some cases for obtaining concentration inequalities for some graph parameter, yet it is not strong enough for transferring fluctuations. A successful approach in this regard for additive parameters is to consider multivariate generating series $A(z,w_1, w_2, \ldots)$ where the additional variables mark the quantities under consideration. For example, $w_i$ could mark the number of vertices of degree~$i$. The strategy is then to determine equations for $A(z,w)$ that yield a central limit theorem for the corresponding functional of $\mA_n$, and transfer these equations via the dissymmetry equality to bivariate equations for $F(z,w)$, which by employing singularity analysis then yield the same central limit theorem for the functional of~$\mF_n$. 

A minor drawback to this generating function approach is that we need to perform this transfer individually for each functional  under consideration. A more severe drawback is that sometimes this transfer is not feasible, for example when we want to prove functional convergence of some contour function to a limiting stochastic process, or  Gromov--Hausdorff--Prokhorov convergence to a limiting random real tree. This problem was noted for example by Haas and Miermont~\cite{MR3050512}, who established Aldous' Brownian tree as scaling limit of random P\'olya trees, and noted that at that time the scaling limit of unlabelled unrooted trees was still an open problem.  In order to address this issue, an approach that allows a transfer of ``practically all''  properties of $\mA_n$ to $\mF_n$ was  presented by S.~\cite[Thm. 1.3]{MR3983790}. Using the mentioned method of cycle pointing~\cite{MR2810913} it was shown that there exists a random tree $\mB_n$ with stochastically bounded size $K_n = O_p(1)$ so that $\mB_n$ is independent from $(\mA_k)_{k \ge 1}$ and such that the unrooted tree $\mA_{n - K_n} + \mB_n$ obtained by identifying the root vertices of $\mA_{n - K_n}$ and  $\mB_n$ satisfies
\[
\lim_{n \to \infty} d_{\mathrm{TV}}(\mA_{n - K_n} + \mB_n, \mF_n) = 0. 
\]
Here $d_{\mathrm{TV}}$ denotes the total variation distance. That is, for any two random variables $X$ and $Y$ with values in a common measurable space it is given by \[
d_{\mathrm{TV}}(X,Y) = \sup_\cE \left| \Pr{X \in \cE} - \Pr{Y \in \cE} \right|,
\]
with the index $\cE$ ranging over all measurable subsets of that space. Thus, the random free tree $\mF_n$ looks like a slightly smaller random P\'olya tree $\mA_{n - K_n}$, with some tree attached to its root to make up for the missing vertices. For all practical purposes, the influence of a stochastically bounded section of a tree is so small that it may be ignored, hence allowing the transfer of both stochastic process convergence as well as central limit theorems for graph parameters from $\mA_n$ to $\mF_n$.

Although for practical purposes the issue of transferring results from $\mA_n$ to $\mF_n$ has hence been resolved, it is somewhat unsatisfactory that the statement of this approximation theorem is rather cumbersome, in particular in light of the simpler and even stronger principle we derive in the present work:

\begin{theorem}
	\label{te:main2}
	It holds that
	\[
	\lim_{n \to \infty} d_{\mathrm{TV}}(F(\mA_n), \mF_n) = 0.
	\]
	More precisely, for each $1/2 < \alpha < 1$ there exist constants $c,C>0$ such that for all $n \ge 1$ and all sets $\cE$
	\[
	|\Pr{F(\mA_n) \in \cE} - \Pr{\mF_n \in \cE}| \le \exp(-cn^{2 \alpha -1}) + \Pr{F(\mA_n) \in \cE} C n^{\alpha-1}.
	\]
\end{theorem}
In other words,  random P\'olya trees and random free trees are asymptotically equivalent. It appears that this fact hasn't been noted anywhere in the literature so far and furnishes all previous transfer arguments. We also provide extensions of this result  to degree restricted free trees and tree-like classes of graphs in Section~\ref{sec:extension}.

\section{Proof of main theorems}

Given an $n$-element set $V$, the symmetric group $\cS[V]$ on that set operates on the set of trees $\cU[V]$ with vertex set $V$ in a natural way, so that $\sigma.T$ is the tree obtained by relabelling its vertices according to some permutation $\sigma \in \cS[V]$. The orbits of this group operation correspond bijectively to the $n$-vertex free trees. We let $\cF_n$ denote the collection of $n$-vertex free trees, and set
\[
\mathrm{Sym}(\cU)[V] = \{(T, \sigma)  \mid T \in \cU[V], \sigma \in \cS[V], \sigma.T = T\}.
\]
The elements of this set are called \emph{symmetries}.
The cardinality of the orbit of an element from $\cU[V]$ equals the index of its stabilizer. Hence the function
\[
\phi: \mathrm{Sym}(\cU)[V]  \to \cF_n
\]
that maps a tree $T$ paired with one of its automorphisms $\sigma$ to the corresponding free tree has the property, that
\[
|\phi^{-1}(\{F\})| = n!
\]
for each free tree $F \in \cF_n$. This well-known fact~\cite{MR633783} lies at the heart of the theory of cycle index series, and ensures that the exponential generating series 
\[
\mathrm{Sym}(\cU)(z) = \sum_{n \ge 1} \frac{1}{n!} |\mathrm{Sym}(\cU)[\{1, \ldots, n\}]| z^n
\]
satisfies
\[
\mathrm{Sym}(\cU)(z) = F(z).
\]
For each integer $k \ge 0$ we let \[
\mathrm{Sym}_{k}(\cU)[V] \subset \mathrm{Sym}(\cU)[V]
\]
denote the subset of all symmetries with precisely $k$ fixed points. We define associated generating series analogously. 

Note that if a symmetry $(T, \sigma)$ has fixed points, then any vertex on the unique path between two fixed points is also a fixed point. Thus, the fixed points of $\sigma$ form a subtree of $T$. Furthermore, the non-fixed point branches attached to a fixed point $v$ are permuted by the symmetry, obviously without fixing any single branch. Thus, we may view $v$ together with all its non-fixed points branches as a symmetry of a rooted tree that only fixes the root-vertex. 

It follows that
\begin{align}
	F(z) = \mathrm{Sym}_0(\cU)(z) + U(H(z))
\end{align}
with
\[
U(z) = \sum_{n \ge 1} \frac{n^{n-2}}{n!} z^n
\]
the exponential generating series of labelled (unrooted) trees, and
\[
H(z) = z\exp \left( \sum_{i \ge 2} A(z^i) /i \right)
\]
the exponential generating series of symmetries of rooted trees that only fix the root vertex. The equation for $H(z)$ follows from the fact that any such symmetry induces a fixed point free permutation $\nu$ of the branches attached to the root, which may be grouped according to the length of the corresponding cycle of $\nu$. For any $i \ge 2$,  $A(z^i)/i$ enumerates $i$-tuples of (isomorphic) rooted trees equipped with an automorphism that permutes them cyclically, see for example~\cite[Prop. 10]{MR633783}.   Hence $\exp( \sum_{i \ge 2} A(z^i)/i)$ counts unordered collections of such objects for all $i \ge 2$, yielding the stated equation for $H(z)$.

Since $A(\rho) = 1$ we know that
\[
\rho \exp \left( \sum_{i \ge 2} A(\rho^i) /i  \right) = 1/e,
\]
and using $T(z) = z U'(z)$ we obtain
\[
U(1/e) = 1/2.
\]
Let $X_1, X_2, \ldots$ denote independent copies of a random variable $X$ with probability generating series
\[
\Ex{z^X} = \rho z\exp \left(1 + \sum_{i \ge 2} A((\rho z)^i) /i  \right).
\]
Note that $X$ has finite exponential moments. For each $k \ge 0$ we set $S_k = \sum_{i=1}^k X_i$ and let $N$ denote an independent random variable with probability generating series
\[
\Ex{z^N} = 2U(z/e) .
\]
This way,
\begin{align}
	\label{eq:diff5}
	[z^n] \left( F(z) - \mathrm{Sym}_0(\cU)(z) \right) = \frac{1}{2} \Pr{S_N = n} \rho^{-n} .
\end{align}
By~\cite[Thm. 1, (ii), (iii)]{zbMATH07179510} and~\cite[Thm. 1.1]{zbMATH00902764} it follows that
\begin{align*}
	\Pr{S_N = n} \sim \Ex{X}^{-1} \Pr{N= \lfloor n/\Ex{X} \rfloor }
\end{align*}
as $n \to \infty$. Hence, by Stirling's formula
\begin{align}
	\label{eq:almost}
	[z^n] \left( F(z) - \mathrm{Sym}_0(\cU)(z) \right) \sim \frac{1}{\sqrt{2\pi}} \Ex{X}^{3/2} n^{-5/2} \rho^{-n}.
\end{align}
The centre of a tree is obtained by simultaneously removing all leaves over and over again until we are left with either an edge or a vertex. The centre is left invariant by any automorphism. Thus, a symmetry with no fixed points must have an edge as centre, and it must transpose the ends of this edge. Thus, the two trees attached to the ends must be isomorphic. It follows that
\begin{align}
	\label{eq:sym0}
	[z^n] \mathrm{Sym}_0(\cU)(z) \le [z^n]A(z^2)
\end{align}
for all $n \ge 1$. Using $\rho <1$ it immediately follows from~\eqref{eq:almost} that
\begin{align}
	[z^n] F(z) \sim \frac{1}{\sqrt{2\pi}} \Ex{X}^{3/2} n^{-5/2} \rho^{-n}.
\end{align}

The class $\cA$ of P\'olya trees may in fact be enumerated analogously. The only difference is that the class $\cT$ of rooted trees takes the place of the class $\cU$ of unrooted trees, and symmetries are required to fix the root vertex. In particular, each symmetry of $\cA$ has at least one fixed point. Hence
\[
A(z) = \mathrm{Sym}(\cT)(z) = T\left( z\exp \left( \sum_{i \ge 2} A(z^i) /i  \right)\right)
\]
and 
\[
[z^n] A(z) = \Pr{S_{\tilde{N}} = n} \rho^{-n}
\]
with $\tilde{N}$ denoting an independent random variable with probability generating function
\[
\Ex{z^{\tilde{N}}} = T(z/e).
\]
By~\cite[Thm. 1, (ii), (iii)]{zbMATH07179510} and~\cite[Thm. 1.1]{zbMATH00902764} we have again
\[
\Pr{S_{\tilde{N}} =n} \sim \Ex{X}^{-1} \Pr{\tilde{N} = \lfloor n / \Ex{X}\rfloor}
\]
and thus
\begin{align}
	[z^n] A(z) \sim \frac{1}{\sqrt{2\pi}} \Ex{X}^{1/2} n^{-3/2} \rho^{-n}.
\end{align}
Hence, we recover Equations~\eqref{eq:polya} and~\eqref{eq:free} with 
\begin{align}
	c_A = \frac{1}{\sqrt{2\pi}} \Ex{X}^{1/2} \qquad \text{and} \qquad c_F = \frac{1}{\sqrt{2\pi}} \Ex{X}^{3/2} .
\end{align} This concludes the proof of Theorem~\ref{te:main1}.

We proceed with the proof of Theorem~\ref{te:main2}. Let $1/2 < \alpha < 1$ be given. Our rough strategy is to bound the probability for a symmetry's number of fixed points to  deviate more than $n^\alpha$ from its asymptotic expectation, and to compare probabilities for rooted and unrooted trees when the fixed points do lie in the correct window.

The random free tree $\mF_n$ may be generated by taking a uniformly selected symmetry from the set $\mathrm{Sym}(\cU)[\{1, \ldots, n\}]$ instead. By Equation~\eqref{eq:sym0} the probability for this symmetry to have no fixed points is given by
\[
\frac{[z^n] \mathrm{Sym}_0(\cU)(z)}{f_n} = O(\rho^{n/2}).
\]
By construction of $N$ and $S_N$, for any $1 \le k \le n$ the probability $\Pr{ N = k \mid S_N = n}$ is equal to the probability to observe exactly $k$ fixed points in a random $n$-sized symmetry that has been conditioned to exhibit at least one fixed point. Thus
\[
\Pr{ N = k \mid S_N = n} = \frac{[z^n]\mathrm{Sym}_k(\cU)(z) }{f_n - [z^n] \mathrm{Sym}_0(\cU)(z)  }.
\]
Hence, using Equation~\eqref{eq:diff5}, it follows that
\[
\frac{[z^n] \mathrm{Sym}_k(\cU)(z)}{f_n} = \Pr{ N = k , S_N = n} \frac{1}{2 f_n  \rho^n}.
\]
Hence
\[
\sum_{\substack{0 \le k \le n \\ |k - n/\Ex{X}| \ge n^\alpha}} \frac{[z^n] \mathrm{Sym}_k(\cU)(z)}{f_n} = O(\rho^{n/2}) + O(n^{5/2}) \sum_{\substack{1 \le k \le n \\ |k - n/\Ex{X}| \ge n^\alpha}} \Pr{S_k = n}.
\]
Using Lemma~\ref{le:meddeviation} it follows that
\begin{align}
	\label{eq:hl2}
	\sum_{\substack{0 \le k \le n \\ |k - n/\Ex{X}| \ge n^\alpha}} \frac{[z^n] \mathrm{Sym}_k(\cU)(z)}{f_n} \le \exp( - \Theta(n^{2\alpha-1})).
\end{align}
By analogous arguments (or by combining the rough bounds~\eqref{eq:rough}
with Inequality~\eqref{eq:hl2}), a similar bound holds for the symmetries of rooted trees:
\[
\sum_{\substack{1 \le k \le n \\ |k - n/\Ex{X}| \ge n^\alpha}} \frac{[z^n] \mathrm{Sym}_k(\cT)(z)}{a_n} \le \exp( - \Theta(n^{2\alpha-1})).
\]
For P\'olya trees, bounds and limits for the number of fixed points have also been carried out in previous works~\cite{MR3773800,MR3853863,MR3764337}. They may also be expressed as component counts of a Gibbs partition in the dense regime~\cite{https://doi.org/10.48550/arxiv.2204.06982}.

For any set $\cE$ of $n$-vertex free trees we can form the corresponding set $\cE'$ of $n$-vertex symmetries, and the subset $\cE'' \subset \cE'$ of these symmetries $S$ whose number $r(S)$ of fixed points satisfies $|k - n/\Ex{X}|< n^\alpha$. Hence
\begin{align*}
	|\Pr{\mF(A_n) \in \cE} - \Pr{\mF_n \in \cE}| &\le  \exp( - \Theta(n^{2\alpha-1})) + \sum_{S \in \cE''} \left| \frac{r(S)}{n!a_n} - \frac{1}{n!f_n} \right|. 
\end{align*}
Using the second order asymptotics~\cite{MR3817525} for P\'olya trees we obtain \begin{align}
	\label{eq:second}
	a_n/f_n = n/\Ex{X} + O(1),
\end{align} and it follows that uniformly for all $S \in \cE''$
\[
\left| \frac{r(S)}{n!a_n} - \frac{1}{n!f_n} \right| =\frac{r(S)}{n!a_n}  O(n^{\alpha-1}).
\]
Hence
\begin{align*}
	|\Pr{\mF(A_n) \in \cE} - \Pr{\mF_n \in \cE}| &\le  \exp( - \Theta(n^{2\alpha-1})) +  O(n^{\alpha-1}) \sum_{S \in \cE''} \frac{r(S)}{n!a_n} \\
	&\le \exp( - \Theta(n^{2\alpha-1})) +  O(n^{\alpha-1}) \Pr{\mF(A_n) \in \cE}.
\end{align*}
This concludes the proof of Theorem~\ref{te:main2}.

\section{Extensions}
\label{sec:extension}

\subsection{Trees with degree restrictions}

Let $\Omega \subset \ndN = \{1, 2, \ldots\}$ denote a fixed set of positive integers such that $1 \in \Omega$ and $k \in \Omega$ for at least one $k \ge 3$. Since we already treated case with no degree restrictions, we additionally assume $\Omega \ne \ndN$. 

Set $\Omega^* = \{i-1 \mid i \in \Omega\}$.  We would like to compare the uniform random free tree $\mF_n^\Omega$ with vertex degrees in $\Omega$ with the uniform random P\'olya tree $\mA_n^{\Omega^*}$ with vertex outdegrees in $\Omega^*$. It is clear  that $\mA_n^{\Omega^*}$ is well defined for $n \equiv 1 \mod \gcd \Omega^*$ large enough, and $\mF_n^\Omega$ is well-defined for $n \equiv 2 \mod \gcd \Omega^*$ large enough. In the following we implicitly only consider values of $n$ where the corresponding model makes sense.

S.~\cite[Thm. 1.3]{MR3983790} proved  that indeed  $\mF_n^\Omega$ admits an approximation in total variation by an $(n - O_p(1))$-vertex random P\'olya tree with outdegrees in $\Omega^*$ and a small tree attached to its root to reach total vertex size $n$.

However, our assumption $\Omega \ne \ndN$ ensures that there exists an integer $k_0 \ge 1$ with $k_0 \in \Omega^* \setminus \Omega$. The limiting probability for the root of $\mA_n^{\Omega^*}$ to have degree $k_0$ is positive, whereas $\mF_n^\Omega$ by construction has no vertex with degree $k_0$. Thus, $\mF_n^\Omega$ and $\mA_n^{\Omega^*}$ are not asymptotically equivalent, and S.~\cite[Thm. 1.3]{MR3983790} already provides the best comparison.

In order to achieve an approximation in total variation we need to define  a random P\'olya tree $\tilde{\mA}_n^{\Omega}$ whose vertex degrees are required to lie in $\Omega$, so that the root has outdegree in $\Omega$ and all other vertices have outdegree in $\Omega^*$.

\begin{proposition}
	\label{pro:degree}
	It holds that
	\[
	\lim_{n \to \infty} d_{\mathrm{TV}}(F(\tilde{\mA}_n^\Omega), \mF_n^\Omega) = 0.
	\]
	For each $1/2 < \alpha < 1$ there exist constants $c,C>0$ such that for all $n \ge 1$ and all sets $\cE$
	\[
	|\Pr{F(\tilde{\mA}_n^\Omega) \in \cE} - \Pr{\mF_n^\Omega \in \cE}| \le \exp(-cn^{2 \alpha -1}) + \Pr{F(\tilde{\mA}_n^\Omega) \in \cE} C n^{\alpha-1}.
	\]
\end{proposition}
\begin{proof}
	Due to the strong similarity with the unrestricted case we only sketch the necessary steps. The probability that a random $\cF^\Omega$-symmetry with $n$ vertices has no fixed points is exponentially small by analogous arguments as in the unrestricted case. The medium deviation concentration inequality for the number of fixed points in an ${\cA}^{\Omega^*}$-symmetry was already noted in~\cite{MR3773800} and~\cite{MR3853863}. The branching process methods there may be easily adapted to get the same bound for ${\tilde{\cA}}^{\Omega}$-symmetries. We use branching process methods here in order to elegantly avoid having to deal with the  restrictions imposed by vertex degrees in the fixed point subtree to   the sum of the  root degrees in the attached branches. The same concentration inequality for unrooted trees then follows by using rough bounds as in~\eqref{eq:rough}.  Hence the rest of the proof is then fully analogous to the unrestricted case. The  required bound on the second order term in the asymptotic expansion for the number of rooted trees required to generalize Equation~\eqref{eq:second} follows readily by transfer theorems  from the known square-root singular expansion.
\end{proof}

The branches attached to the root of $\tilde{\mA}_n^{\Omega}$ form a multiset of degree restricted P\'olya trees, with the total number constrained to lie in $\Omega$. The distribution of the component sizes is hence a special case of an unlabelled Gibbs partition (with all weights set to $1$). Using the main result of~\cite{MR4079638} (slightly adapted to take into account periodicities) it follows that precisely one of the branches attached to the root of $\tilde{\mA}_n^{\Omega}$  is asymptotically macroscopic, with the rest having a stochastically bounded total size. The distribution of the largest branch is hence that of $\mA_{n - K_n}^{\Omega^*}$ for some independent stochastically bounded non-negative integer $K_n$. In other words, $\tilde{\mA}_n^{\Omega}$ (and by Proposition~\ref{pro:degree} also $\mT_n^\Omega$) therefore looks like an $(n - O_p(1))$-vertex random P\'olya tree with outdegrees in $\Omega^*$ and a small tree attached to its root to reach total vertex size $n$. Thus, we recover the main result of S.~\cite[Thm. 1.3]{MR3983790} in a different way without using the cycle pointing method.

\subsection{Subcritical classes of graphs}


The definition of a subcritical class of unlabelled graphs~\cite[Def. 10]{MR2873207} involves the concept of cycle index series $Z_G$ of a graph class $\cG$, as well as operations such as rooted and derived classes of structures. The cycle index series of $\cG$ is a power series in countably infinitely many variables
\[
Z_G(s_1, s_2, \ldots) = \sum_{n \ge 0} \frac{1}{n!} \sum_{(G, \sigma) \in \mathrm{Sym}(\cG)[\{1, \ldots, n\}]} s_1^{\sigma_1} s_2^{\sigma_2} \cdots 
\]
with symmetries of graphs being defined analogously as for trees, and  $\sigma_i$ denoting the number of cycles of length $i$ in $\sigma$ for all $i \ge 1$. Likewise, for rooted graphs all symmetries are required to fix the root vertex, and the cycle index series for classes of rooted graphs are defined accordingly. Derived graphs are rooted graphs, except that the root-vertex is required to have a place-holder label $*$. That is, for a derived class of graphs the symmetries on a set of vertices $U$ are pairs $(G, \sigma)$ with $G$ a rooted graph on the vertex set $U \cup \{*\}$ (assuming always $* \notin U$), and $\sigma: U \to U$ a permutation of $U$ such that the canonical bijective extension $U \cup \{*\} \to U \cup\{*\}$ is a graph automorphism of $G$. We refer the reader to~\cite{MR633783,MR1629341} for details on the technical background of these notions.

A class of graphs $\cG$ (assumed to be closed under relabelling of vertices) is called block-stable, if it contains the graph consisting of a single vertex, and any graph lies in $\cG$ if and only if all its blocks (maximal $2$-connected subgraphs) lie in $\cG$. Its subclasses of connected and $2$-connected graphs are denoted by $\cC$ and $\cG$. We let $\cC^\bullet$ and $\cB'$ denote the corresponding classes of vertex rooted and derived graphs. The ordinary generating series and cycle index sums of these classes are related by
\begin{align*}
	G(z) &= \exp\left(\sum_{i \ge 1} C(z^i)/i\right), \\
	C^\bullet(z) &= z \exp\left(\sum_{i \ge 1} Z_{B'}(C^\bullet(z^i), C^\bullet(z^{2i}, C^\bullet(z^{3i}), \ldots )/i\right).
\end{align*}
The class is termed subcritical in the unlabelled case~\cite[Def. 10]{MR2873207}, if additionally the following conditions are met.
\begin{enumerate}
	\item $C^\bullet(z)$ as positive radius of convergence $\rho_C>0$.
	\item $Z_{B'}(y, C^\bullet(z^{2}), C^\bullet(z^{3}), \ldots )$ is analytic at $(C^\bullet(\rho_C), \rho_C)$.
	\item $\sum_{i \ge 2} Z_{B'}(C^\bullet(z^i), C^\bullet(z^{2i}, C^\bullet(z^{3i}), \ldots )/i$ has radius of convergence strictly larger than $\rho_C$.
	\item $Z_{C}(0, z^2, z^3, \ldots)$ has radius of convergence strictly larger than $\rho_C$.
\end{enumerate}
Subcriticality for labelled graphs is defined differently, and the relationship whether subcriticality in the labelled case implies subcriticality in the unlabelled case or vice versa has not been investigated so far.

Subcritical classes of graphs are sometimes called tree-like, as they admit the Brownian tree as scaling limit and therefore their global  shape resembles that of a tree. This was shown for the labelled case in~\cite{MR3551197}, the unlabelled rooted case in~\cite{MR3853863}, and the unlabelled unrooted case in~\cite{MR4397030}.

In the following we assume that the graph class $\cG$ is subcritical in the unlabelled case. We let $\mC_n$ and $\mC_n^\bullet$ denote the random unlabelled unrooted connected graph and the random unlabelled rooted connected graph with $n$ vertices from the classes $\cC$ and $\cC^\bullet$ of unlabelled graphs. We let $F(\mC_n^\bullet)$ denote the unrooted graph obtained by forgetting which vertex is marked.

An earlier approximation result~\cite[Thm. 1]{MR4397030} formulated for block-weighted showed that $\mC_n$ may be approximated  in total variation by an $(n - O_p(1))$-sized random rooted graph with a small graph attached to its root to make up for the missing number of vertices. While this is already sufficient for all practical purposes, the following approximation result provides  a more elegant comparison:

\begin{proposition}
	\label{pro:subcrit}
	It holds that
	\[
	\lim_{n \to \infty} d_{\mathrm{TV}}(F(\mC_n^\bullet), \mC_n) = 0.
	\]
	For each $1/2 < \alpha < 1$ there exist constants $c,C>0$ such that for all $n \ge 1$ and all sets $\cE$
	\[
	|\Pr{F(\mC_n^\bullet) \in \cE} - \Pr{\mC_n \in \cE}| \le \exp(-cn^{2 \alpha -1}) + \Pr{F(\mC_n^\bullet) \in \cE} C n^{\alpha-1}.
	\]
\end{proposition}
\begin{proof}
	Setting the first variable in a cycle index sum to zero means discarding all symmetries with no fixed point, hence
	\[
	\mathrm{Sym}_0(\cC)(z) = Z_C(0, z^2, z^3, \ldots).
	\]
	Consequently, the fourth condition in the definition of subcriticality ensures that only an exponentially small fraction of symmetries of $\cC$ have no fixed-points. Furthermore, the concentration inequality for the number of fixed points in a symmetry of $\cC^\bullet$ was already obtained in~\cite{MR3853863} via branching process methods. The bound on the second order term in the asymptotics for the number of rooted graphs required to generalize Equation~\eqref{eq:second} follows readily from the square-root singular expansions given in~\cite{MR2873207}. Thus, the proof may be carried out entirely analogously as the proof of Theorem~\ref{te:main2}.
\end{proof}

\section*{Appendix}

The following deviation inequality may be  found in most textbooks on the subject.

\begin{lemma}[Medium deviation inequality for one-dimensional random walk]
	\label{le:meddeviation}
	Let $(X_i)_{i \in \ndN}$ be an i.i.d. family of real-valued random variables with $\Ex{X_1} = 0$ and $\Ex{e^{t X_1}} < \infty$ for all $t$ in some open interval containing zero. Then there are constants $\delta, c>0$ such that for all $n\in \ndN$, $x > 0$ and $0 \le\lambda\le\delta$ it holds that \[\Pr{|X_1 + \ldots + X_n| \ge x} \le 2 \exp(c n \lambda^2 - \lambda x).\]
\end{lemma}

\section*{Conflict of Interest Statement}
The corresponding author states that there is no conflict of interest. 

\section*{Data Availability Statement}

There is no data associated to this manuscript.

\acknowledgements
\label{sec:ack}
I warmly thank the referees for their helpful remarks. This research was funded in whole or
in part by the Austrian Science Fund (FWF) [10.55776/PAT6732623]. For open access purposes, the author
has applied a CC BY public copyright license to any author accepted manuscript version arising from this
submission.

\bibliographystyle{abbrvnat}
\bibliography{sample-dmtcs}
\label{sec:biblio}

\end{document}